\documentclass[11pt]{amsart}
\usepackage{amsmath,amsthm,amssymb, amscd, amsfonts}
\usepackage[all]{xy}
\usepackage{leftidx}
\usepackage[latin1]{inputenc}
\usepackage{enumerate}
\usepackage[dvipsnames]{xcolor}

\theoremstyle{plain}

\newtheorem{theorem}{Theorem}[section]
\newtheorem{corollary}[theorem]{Corollary}

\newtheorem{proposition}[theorem]{Proposition}
\newtheorem{lemma}[theorem]{Lemma}

\theoremstyle{definition}
\newtheorem{definition}[theorem]{Definition}

\theoremstyle{remark}
\newtheorem{remark}[theorem]{Remark}

\numberwithin{equation}{section}\theoremstyle{plain}

\newcommand{\I}{\mathcal{I}}

\renewcommand{\1}{\textbf{1}}

\newcommand{\A}{{\mathcal A}}
\newcommand{\B}{{\mathcal B}}
\newcommand{\C}{{\mathcal C}}
\newcommand{\D}{{\mathcal D}}

\newcommand{\Z}{{\mathcal Z}}

\newcommand{\Rep}{\operatorname{Rep}}

\newcommand\Irr{\operatorname{Irr}}
\newcommand\FPdim{\operatorname{FPdim}}

\newcommand\vect{\operatorname{Vec}}
\newcommand\svect{\operatorname{sVec}}
\newcommand\id{\operatorname{id}}

\newcommand\Hom{\operatorname{Hom}}

\begin{document}
\title[Generalized near-group fusion categories]{On generalized near-group fusion categories}
\author{Jingcheng Dong}
\email{jcdong@nuist.edu.cn}
\address{College of Mathematics and Statistics, Nanjing University of Information Science and Technology, Nanjing 210044, China}




\date{\today}

\begin{abstract}
In this paper, we study the structure of a generalized near-group fusion category and classified it when it is slightly degenerate.
\end{abstract}

\maketitle

\section{Introduction}\label{sec1}
Let $\C$ be a fusion category, and let $G$ be the group generated by invertible simple objects of $\C$. Then there is an action of $G$ on the set of non-isomorphic non-invertible simple objects by left tensor product. If this action is transitive then $\C$ is called a generalized near-group fusion category in \cite{Thornton2012Generalized}. In his thesis \cite{Thornton2012Generalized}, Thornton proved that $\C$ is $\varphi$-pseudounitary and classified $\C$ when it is symmetric or modular.

\medbreak
Let $\C$ be a generalized near-group fusion category. Then for every non-invertible simple object $X$, $X\otimes X^*$ admits the same decomposition (see Section \ref{sec3}):
\begin{equation}
\begin{split}
X\otimes X^*=\bigoplus_{h\in \Gamma}h\oplus k_1X_{1}\oplus \cdots\oplus k_nX_{n},
\end{split}\nonumber
\end{equation}
where $\{X_1,\cdots,X_n\}$ is a full list of non-isomorphic non-invertible simple objects of $\C$, $\Gamma$ is the stabilizer of $X$ under the action of $G$. In this paper, we shall say that $\C$ is a generalized near-group fusion category of type $(G,\Gamma,k_1,\cdots,k_n)$.  If $(k_1,\cdots,k_n)=(0,\cdots,0)$ then $\C$ is a generalized Tambara-Yamagami fusion category introduced in \cite{liptrap2010generalized}. If $\C$ exactly has one non-invertible simple object, then $G=\Gamma$ and $\C$ is a near-group fusion category introduced in \cite{siehler2003near}. The main goal of this paper is to study the structure of $\C$ and classify it when it is slightly degenerate.

\medbreak
The paper is organized as follows. In Section \ref{sec2}, we recall some basic results and prove some basic lemmas which will be used throughout.

In Section \ref{sec3}, we study the fusion rules, non-pointed fusion subcategories of a generalized near-group fusion category $\C$. In particular, we obtain that every component $\C_g$ of the universal grading exactly contains the simple objects $\alpha_g, \alpha_g\otimes Y_1, \cdots, \alpha_g\otimes Y_s $, where $\alpha_g$ is an invertible simple object in $\C_g$ and $\1,Y_{1},\cdots,Y_{s}$ is a list of all nonisomorphic simple object in the adjoint subcategory $\C_{ad}$.

In Section \ref{sec4}, we study the slightly degenerate generalized near-group fusion categories. Our result shows that slightly degenerate generalized near-group fusion categories fit into four classes.

\section{Preliminaries}\label{sec2}
A fusion category $\C$ is a $\mathbb{C}$-linear semisimple rigid tensor category with finitely many isomorphism classes of simple objects, finite-dimensional vector space of morphisms and the unit object $\1$ is simple.

\subsection{Invertible simple objects}\label{sec2.1}
Let $\C$ be a fusion category. The tensor product in $\C$ induces a ring structure on the Grothendieck ring $K(\C)$. By \cite[Section 8]{etingof2005fusion}, there is a unique ring homomorphism $\FPdim:K(\C)\to\mathbb{R}$ such that $\FPdim(X)\geq 1$ for all nonzero $X\in \C$. We call $\FPdim(X)$ the Frobenius-Perron dimension of $X$. The Frobenius-Perron dimension of $\C$ is defined by
$\FPdim(\C)=\sum_{X\in\Irr(\C)}\FPdim(X)^2$, where $\Irr(\C)$ is the set of isomorphism classes of simple objects in $\C$.

A simple object $X\in \C$ is called invertible if $X\otimes X^*\cong \1$, where $X^*$ is the dual of $X$. This implies that $X$ is invertible if and only if $\FPdim(X)=1$. A fusion category $\C$ is called pointed if every element in $\Irr(\C)$ is invertible. Let $\C_{pt}$ be the fusion subcategory generated by all invertible simple objects in $\C$. Then $\C_{pt}$ is the largest pointed fusion subcategory of $\C$.

Let $G(\C)$ be the group generated by $\Irr(\C_{pt})$. Then $G(\C)$ admits an action on the set $\Irr(\C)$ by left tensor product. Let $G[X]$ be the stabilizer of any $X\in \Irr(\C)$ under this action. Hence for any simple object $X$, we have a decomposition
\begin{equation}\label{decom1}
\begin{split}
X\otimes X^*=\bigoplus_{g\in G[X]}g\oplus\sum_{Y\in \Irr(\C)/G[X]}  \dim\Hom(Y,X\otimes X^*)Y.
\end{split}
\end{equation}

\subsection{Group extensions of fusion categories}\label{sec2.2}
Let $G$ be a finite group. A fusion category $\C$ is graded by $G$  if $\C$ has a direct sum of full abelian subcategories $\C=\oplus_{g\in G}\C_g$ such that $(\C_g)^*=\C_{g-1}$ and $\C_g\otimes\C_h\subseteq\C_{gh}$ for all $g,h\in G$. If $\C_g\neq 0$ for any $g\in G$ then this grading is called faithful. If this is the case we say that $\C$ is a $G$-extension of the trivial component $\C_e$.

If $\C=\oplus_{g\in G}\C_g$ is faithful then \cite[Proposition 8.20]{etingof2005fusion} shows that
\begin{equation}\label{FPdimgrading}
\begin{split}
\FPdim(\C_g)=\FPdim(\C_h),\,\, \FPdim(\C)=|G| \FPdim(\C_e), \forall g,h\in G.
\end{split}
\end{equation}

It follows from \cite{gelaki2008nilpotent} that every fusion category $\C$ has a canonical faithful grading $\C=\oplus_{g\in \mathcal{U}(\C)}\C_g$ with trivial component $\C_e=\C_{ad}$, where $\C_{ad}$ is the adjoint subcategory of $\C$ generated by simple objects in $X\otimes X^*$ for all $X\in \Irr(\C)$. This grading is called the universal grading of $\C$, and $\mathcal{U}(\C)$ is called the universal grading group of $\C$.

\subsection{M\"{u}ger centralizer}\label{sec2.3}
A braided fusion category $\C$ is a fusion category admitting a braiding $c$, where the braiding is a family of natural isomorphisms: $c_{X,Y}$:$X\otimes Y\rightarrow Y\otimes X$ satisfying the hexagon axioms for all $X,Y\in\C$.

Let $\D$ be a fusion subcategory of a braided fusion category $\C$. Then the M\"{u}ger centralizer $\D'$ of $\D$ in $\C$ is the fusion subcategory generated by
$$\D'=\{Y\in\C|c_{Y,X}c_{X,Y}=\id_{X\otimes Y}\, \mbox{for all}\, X\in\D\}.$$
The M\"{u}ger center $\mathcal{Z}_2(\C)$ of $\C$ is the M\"{u}ger centralizer $\C'$ of $\C$.

\begin{definition}
A braided fusion category $\C$ is called non-degenerate if its M\"{u}ger center $\mathcal{Z}_2(\C)=\vect$ is trivial.
\end{definition}

The following theorem implies that a braided fusion category containing a non-degenerate subcategory admits a decomposition in terms of Deligne tensor product. In the case when $\C$ is modular, it is due to M\"{u}ger \cite[Theorem 4.2]{muger2003structure}

\begin{theorem}{\cite[Theorem 3.13]{drinfeld2010braided}}\label{MugerThm}
Let $\C$ be a braided fusion category and $\D$ be a non-degenerate subcategory of $\C$. Then $\C$ is braided equivalent to $\D\boxtimes \D'$, where $\boxtimes$ stands for the Deligne tensor product.
\end{theorem}

A braided fusion category $\C$ is called symmetric if $\mathcal{Z}_2(\C)=\C$.  A symmetric fusion category $\C$ is called Tannakian if there exists a finite group $G$ such that $\C$ is equivalent to $\Rep(G)$ as braided fusion categories.

By \cite[Corollary 2.50]{drinfeld2010braided}, a symmetric fusion category $\C$ is a $\mathbb{Z}_2$-extension of its maximal Tannakian subcategory. In particular, if $\FPdim(\C)$ is odd then $\C$ is automatically Tannakian.

Symmetric categories are completely degenerate categories, while non-degenerate fusion categories are completely non-degenerate. Between these two extremes, we also consider the following case.

\begin{definition}
A braided fusion category $\C$ is called slightly degenerate if its M\"{u}ger center $\mathcal{Z}_2(\C)$ is equivalent, as a symmetric category, to the category $\svect$ of super vector spaces.
\end{definition}

\begin{lemma}{\cite[Proposition 2.5]{Dong2018extensions}}\label{Cpt_of_sligdegen}
Let $\C$ be a slightly degenerate braided fusion category. Then one of the following holds true.

(1)\, $\FPdim(\C_{pt})=|\mathcal{U}(\C)|$ and $\mathcal{Z}_2(\C) \nsubseteq \C_{ad}$.

(2)\, $\FPdim(\C_{pt})=2|\mathcal{U}(\C)|$ and $\mathcal{Z}_2(\C)\subseteq\mathcal{Z}_2(\C_{ad})=\mathcal{Z}_2(\C_{ad}^{'})$.
\end{lemma}

Let $\Irr_{\alpha}(\C)$ be the set of non-isomorphic simple objects of Frobenius-Perron dimension $\alpha$.

\begin{lemma}\label{slight_degenerate}
Let $\C$ be a braided fusion category. Suppose that the M\"{u}ger center $\mathcal{Z}_2(\C)$ contains the category $\svect$ of super vector spaces. Then the cardinal number of $\Irr_{\alpha}(\C)$ is even for every $\alpha$.
\end{lemma}
\begin{proof}
Let $\delta$ be the invertible object generating $\svect$, and let $X$ be an element in $\Irr_{\alpha}(\C)$. Then $\delta\otimes X$ is also an element in $\Irr_{\alpha}(\C)$. By \cite[Lemma 5.4]{muger2000galois}, $\delta\otimes X$ is not isomorphic to $X$. This implies that $\Irr_{\alpha}(\C)$ admits a partition $\{X_1,\cdots,X_n\}\cup \{\delta\otimes X_1,\cdots,\delta\otimes X_n\}$. Hence the cardinal number of $\Irr_{\alpha}(\C)$ is even.
\end{proof}

\subsection{Exact factorizations of fusion categories}
Let $\C$ be a fusion category, and let $\A, \B$ be fusion subcategories of $\C$. Let $\A\B$ be the full abelian (not necessarily tensor) subcategory of $\C$ spanned by direct summands in $X\otimes Y$, where $X\in \A$ and $Y\in \B$. We say that $\C$ factorizes into a product of $\A$ and $\B$ if $\C=\A\B$. A factorization $\C=\A\B$ of $\C$ is called exact if $A\cap \B=\vect$, and is denoted by $\C=\A\bullet\B$, see \cite{gelaki2017exact}.

By \cite[Theorem 3.8]{gelaki2017exact}, $\C=\A\bullet\B$ is an exact factorization if and only every simple object of $\C$ can be uniquely expressed in the form $X\otimes Y$, where $X\in \Irr(\A)$ and $\Irr(\B)$.

\section{Structure of a generalized near-group fusion category}\label{sec3}
In the rest of this paper, we assume that the fusion categories involved is not pointed, since pointed fusion categories have been classified, see e. g. \cite{Ostrik2003}.

\medbreak
Let $\C$ be a fusion category. Recall from Section \ref{sec2.1} that $G:=G(\C)$ acts on $\Irr(\C)$ by left tensor product.

\begin{definition}
A generalized near-group fusion category is a fusion category $\C$ such that $G$ transitively acts on the set $\Irr(\C)/G$.
\end{definition}

Let $\C$ be a generalized near-group fusion category and let $\Irr(\C)/G=\{X_1,\cdots,X_n\}$ be a full list of non-isomorphic non-invertible simple objects of $\C$. By equation \ref{decom1}, we may assume
\begin{equation}\label{decom2}
\begin{split}
X_1\otimes X_1^*=\bigoplus_{h\in \Gamma}h\oplus k_1X_{1}\oplus \cdots\oplus k_nX_{n},
\end{split}
\end{equation}
where $\Gamma=G[X_1]$ is the stabilizer of $X_1$ under the action of $G$, $k_1,\cdots, k_n$ are non-negative integers.

\begin{lemma}\label{fusionrules}
Let $\C$ be a generalized near-group fusion category. Then the fusion rules of $\C$ are determined by:

(1)\, For any $1\leq i\leq n$, we have
\begin{equation}
\begin{split}
X_i\otimes X_i^*=X_1\otimes X_1^*.
\end{split}\nonumber
\end{equation}

(2)\, For any $1\leq i,j\leq n$, there exists $g\in G$ such that
\begin{equation}
\begin{split}
X_i\otimes X_j=\bigoplus_{h\in \Gamma}gh\oplus k_1g\otimes X_{1}\oplus \cdots\oplus k_n g\otimes X_{n}.
\end{split}\nonumber
\end{equation}
\end{lemma}
\begin{proof}
(1)\, Since $G$ transitively acts on $\Irr(\C)/G(\C)$, there exists $g_i\in G$ such that $X_i^*=g_i\otimes X_1^*$ for any $i$. Then
\begin{equation}
\begin{split}
X_i\otimes X_i^*&\cong X_i^{**}\otimes X_i^*\cong(g_i\otimes X_1^*)^*\otimes (g_i\otimes X_1^*)\\
&\cong X_1\otimes g_i^*\otimes g_i\otimes X_1^*\cong X_1\otimes X_1^*.
\end{split}\nonumber
\end{equation}
(2)\, For any $i,j$, there exists $g\in G$ such that $X_i\cong g\otimes X_j^*$. Then
\begin{equation}
\begin{split}
X_i\otimes X_j&\cong g\otimes X_j^*\otimes X_j\cong g\otimes(\bigoplus_{h\in \Gamma}h\oplus k_1X_{1}\oplus \cdots\oplus k_nX_{n})\\
&\cong\bigoplus_{h\in \Gamma}gh\oplus k_1g\otimes X_{1}\oplus \cdots\oplus k_s g\otimes X_{n}.
\end{split}\nonumber
\end{equation}
\end{proof}

Let $G,\Gamma$ and $k_1,\cdots,k_n$ be the data associated to $\C$ as in Lemma \ref{fusionrules}. We shall say $\C$ is a generalized near-group fusion category of type $(G,\Gamma,k_1,\cdots,k_n)$.

\begin{proposition}\label{normalsubgroup}
Let $\C$ be a generalized near-group fusion category of type $(G,\Gamma,k_1,\cdots,k_n)$. Then

(1)\, $\Gamma$ is a normal subgroup of $G$.

(2)\, $\Irr(\C)=G\cup \{X_s|s\in G/\Gamma\}$, where $X_{\overline{g}}=g\otimes X_1$, $g\in G$.

(3)\, The rank of $\C$ is $[G:\Gamma](1+|\Gamma|)$ and $\FPdim(\C)=[G:\Gamma](\FPdim(X)^2+|\Gamma|)$.
\end{proposition}
\begin{proof}
(1)\, By Lemma \ref{fusionrules}, $G[g\otimes X_1]=G[X_1]=\Gamma$ for any $g\in G$. On the other hand, $G[g\otimes X_1]=gG[X_1]g^{-1}=g\Gamma g^{-1}$. Hence $\Gamma$ is normal in $G$.

(2)\, Let $X_{\overline{g}}=g\otimes X_1$ for every $\overline{g}\in G/\Gamma$. Since $\Gamma=G[X_1]$, we have $g\otimes X_1\cong h\otimes X_1$ if and only if $h^{-1}g\otimes X_1\cong X_1$ if and only if $h^{-1}g\in \Gamma$ if and only if $\overline{g}=\overline{h}$ in $G/\Gamma$. Hence the isomorphic class of $X_{\overline{g}}$ is well defined.

(3)\, Part (3) follows from Part (2).
\end{proof}

\begin{remark}\label{T_Y}
Let $\C$ be a generalized near-group fusion category of type $(G,\Gamma,k_1,\cdots,k_n)$.

(1)\, If $(k_1,\cdots,k_n)=(0,\cdots,0)$ then $X_i\otimes X_j$ is a direct sum of invertible simple objects by Lemma \ref{fusionrules}. Then $\C$ is a generalized Tambara-Yamagami fusion category introduced in \cite{liptrap2010generalized}. In fact, it is easily observed that $\C$ is a generalized Tambara-Yamagami fusion category if and only if $(k_1,\cdots,k_n)=(0,\cdots,0)$.

(2)\, If $\C$ exactly has one non-invertible simple object, then $G=\Gamma$ and $\C$ is a near-group fusion category introduced in \cite{siehler2003near}.
\end{remark}

\begin{proposition}\label{subcategory}
Let $\C$ be a generalized near-group fusion category of type $(G,\Gamma,k_1,\cdots,k_n)$. Assume that $\D$ is  a non-pointed fusion subcategory of $\C$. Then $\D$ is also a generalized near-group fusion category.
\end{proposition}
\begin{proof}
We shall prove that $G(\D)$ transitively acts on $\Irr(\D)/G(\D)$. Let $X_i$ and $X_j$ be non-invertible simple objects in $\D$. Then there exists $g\in G$ such that $X_j=g\otimes X_i$. From $\dim\Hom(X_j,g\otimes X_i)=\dim\Hom(g,X_j\otimes X_i^*)=1$, we know that $g$ is a summand of $X_j\otimes X_i^*$. On the other hand, $X_j\otimes X_i^*$ lies in $\D$ since $\D$ is a fusion subcategory of $\C$. Hence $g$ is an element of $G(\D)$. This proves that $G(\D)$ transitively acts on $\Irr(\D)/G(\D)$
\end{proof}

\begin{theorem}\label{categorytype}
Let $\C$ be a generalized near-group fusion category of type $(G,\Gamma,k_1,\cdots,k_n)$. Assume that $(k_1,\cdots,k_n)\neq (0,\cdots,0)$. Then

(1)\, The adjoint subcategory $\C_{ad}$ is non-pointed. There is a 1-1 correspondence between the non-pointed fusion subcategories of $\C$ and the subgroups of the universal grading group $\mathcal{U}(\C)$.

(2)\, For any $g\in \mathcal{U}(\C)$, the component $\C_g$ contains at least one invertible simple object. In particular, $\Irr(\C_g)=\{\alpha_g, \alpha_g\otimes Y_1, \cdots, \alpha_g\otimes Y_s \}$, where $\alpha_g$ is an invertible simple object in $\C_g$ and $\Irr(\C_{ad})=\{\1,Y_{1},\cdots,Y_{s}\}$.
\end{theorem}
\begin{proof}
(1)\, Let $\D$ be a non-pointed fusion subcategory of $\C$. For every non-invertible simple object $X\in \D$, Lemma \ref{fusionrules} shows that
\begin{equation}
\begin{split}
X\otimes X^*=\bigoplus_{h\in \Gamma}h\oplus k_1X_{1}\oplus \cdots\oplus k_nX_{n}.
\end{split}\nonumber
\end{equation}
Hence the adjoint subcategory $\C_{ad}$ is generated by $\Gamma$ and $X_i$'s with $k_i\neq 0$. Since $(k_1,\cdots,k_n)\neq (0,\cdots,0)$, $\C_{ad}$ is not pointed. In particular, $\C_{ad}$ is a fusion subcategory of $\D$. This shows that every non-pointed fusion subcategory of $\C$ contains $\C_{ad}$. Therefore, part (1) follows from \cite[Corollary 2.5]{drinfeld2010braided}.

(2)\, We shall first show that every component $\C_g$ of the universal grading at least contains an invertible simple object. By part (1), $\C_{ad}$ contains a non-invertible simple object $Y$. Let $X$ be a simple object in $\C_g$. We may assume that $X$ is not invertible. Then $X\otimes Y\in \C_g\otimes \C_{ad}\subseteq \C_g$. By Lemma \ref{fusionrules}(2), $X\otimes Y$ contains $|\Gamma|$ invertible simple objects. Hence $\C_g$ contains at least one invertible simple object.

Let $\alpha_g\in\C_g$ be an invertible simple object, and $\1,Y_{1},\cdots,Y_{s}$ be all non-isomorphic simple objects in $\C_{ad}$. Then $\alpha_g,\alpha_g\otimes Y_{1},\cdots,\alpha_g\otimes Y_{s}$ are non-isomorphic simple objects in $\C_g$. Since
$$\FPdim(\alpha_g\otimes Y_{i})=\FPdim(Y_{i})\quad \mbox{and}\quad \FPdim(\C_g)=\FPdim(\C_{ad}),$$
we obtain that $\alpha_g, \alpha_g\otimes Y_1, \cdots, \alpha_g\otimes Y_s$ are all non-isomorphic simple objects in $\C_g$. This completes the proof.
\end{proof}

\begin{remark}\label{remark1}
Let $\C$ be a generalized near-group fusion category of type $(G,\Gamma,k_1,\cdots,k_n)$. Then Proposition \ref{categorytype} implies the following two facts:

(1)\, If $(k_1,\cdots,k_n)\neq (0,\cdots,0)$ then the adjoint subcategory $\C_{ad}$ is the smallest non-pointed fusion subcategory of $\C$. This is because that $\C_{ad}$ corresponds to the trivial subgroup of $\mathcal{U}(\C)$.

(2)\, Assume that $(k_1,\cdots,k_n)\neq (0,\cdots,0)$. Then $\C_{ad}$ is not pointed by Proposition \ref{categorytype}. Let $X\in \C_{ad}$ be a non-invertible simple object. Then Lemma \ref{fusionrules} shows the decomposition of $X\otimes X^*$ contains non-invertible simple objects. Hence $(\C_{ad})_{ad}$ is not pointed. But part (1) shows that $\C_{ad}$ is the smallest non-pointed fusion subcategory of $\C$. Hence $\C_{ad}=(\C_{ad})_{ad}$, and hence the universal grading group $\mathcal{U}(\C_{ad})$ of $\C_{ad}$ is trivial.
\end{remark}

\begin{corollary}\label{main01}
Let $\C$ be a generalized near-group fusion category of type $(G,\Gamma,k_1,\cdots,k_n)$. Assume that $(k_1,\cdots,k_n)\neq (0,\cdots,0)$ and the group $G(\C_{ad})$ is trivial. Then $\C=\C_{pt}\bullet\C_{ad}$ admits an exact factorization of $\C_{pt}$ and $\C_{ad}$.
\end{corollary}

\begin{proof}
 Since $G(\C_{ad})$ is trivial, Theorem \ref{categorytype}(2) shows that every component $\C_g$ exactly contains only one invertible simple object. Let $\delta_g$ be the invertible simple object in $\C_g$. Then $\{\delta_g|g\in U(\C)\}=G(\C)$, and hence every simple object of $\C$ can be expressed in the form $X\otimes Y$, where $X\in \C_{pt}$ and $Y\in \C_{ad}$ are simple objects, also by Theorem \ref{categorytype}(2). The result then follows from \cite[Theorem 3.8]{gelaki2017exact}.
\end{proof}

\section{Slightly degenerate generalized near-group fusion categories}\label{sec4}
Recall from \cite{ostrik2003fusion} that a Yang-Lee category is a rank $2$ modular category which admits the Yang-Lee fusion rules.

\begin{lemma}\label{dimuc}
Let $\C$ be a generalized near-group fusion category of type $(G,\Gamma,k_1,\cdots,k_n)$. Assume that $\FPdim(\C_{pt})=|\mathcal{U}(\C)|$ and $(k_1,\cdots,k_n)\neq(0,\cdots,0)$. Then $\C_{ad}$ is a Yang-Lee category.
\end{lemma}
\begin{proof}
By Proposition \ref{categorytype}, every component $\C_g$ of the universal grading of $\C$ at least has one invertible simple object. Hence, our assumption implies that every component $\C_g$ exactly contains one invertible simple object.

By Proposition \ref{normalsubgroup}, the number of non-isomorphic non-invertible simple objects is not more than the order of $G$. In addition, Theorem \ref{categorytype} shows that every component $\C_g$ admits the same type. Hence every component $\C_g$ only contains two simple objects: one is invertible and the other is not. In particular, $\C_{ad}$ is a Yang-Lee category by the classification of rank $2$ fusion categories \cite{ostrik2003fusion}.
\end{proof}

An Ising category $\I$ is a fusion category which is not pointed and has Frobenius-Perron dimension $4$. Recall from \cite{drinfeld2010braided} that any Ising category $\I$ is a non-degenerate braided fusion category and the adjoint subcategory $\I_{ad}=\I_{pt}$ is braided equivalent to $\svect$.

\begin{lemma}\label{slightly-deg1}
Let $\C$ be a braided generalized near-group fusion category of type $(G,\Gamma,k_1,\cdots,k_n)$. Assume that $(k_1,\cdots,k_n)=(0,\cdots,0)$ and $\C$ is slightly degenerate. Then $\C$ is exactly one of the following:

(1)\,  $\C\cong \I\boxtimes \B$, where $\I$ is an Ising category, $\B$ is a slightly degenerate pointed fusion category.

(2)\, $\C$ is generated by a $\sqrt{2}$-dimensional simple object. In this case, $\C$ is prime.
\end{lemma}
\begin{proof}
Since we assume that $(k_1,\cdots,k_n)=(0,\cdots,0)$, the adjoint subcategory $\C_{ad}$ is generated by $\Gamma$ and $\FPdim(X)=\sqrt{|\Gamma|}$ for all non-invertible simple object $X$ of $\C$. In particular, $\C$ is a generalized Tambara-Yamagami fusion category. By \cite[Proposition 5.2(ii)]{natale2013faithful}, we have
\begin{equation}\label{order}
\begin{split}
|\mathcal{U}(\C)|=2[G:\Gamma].
\end{split}
\end{equation}
By Proposition \ref{Cpt_of_sligdegen}, $|G|=2|\mathcal{U}(\C)|$ or $|G|=|\mathcal{U}(\C)|$.

\medbreak
Case $|G|=2|\mathcal{U}(\C)|$. In this case, equality (\ref{order}) implies that $|\Gamma|=4$. Proposition \ref{Cpt_of_sligdegen} shows that in our case $\C_{ad}$ contains the M\"{u}ger center $\svect$ of $\C$. Let $\delta$ be the invertible simple object generating $\svect$. Then we may write $\Gamma=\{\1,\delta,g,h\}$. Hence $X\otimes X^*=\1\oplus \delta\oplus g\oplus h$ for any non-invertible simple object $X$. In particular, $\dim\Hom(\delta\otimes X,X)=\dim\Hom(\delta,X\otimes X^*)=1$ shows that $\delta\otimes X\cong X$, which contradicts \cite[Proposition 2.6(i)]{etingof2011weakly}. So we can discard this case.

\medbreak
Case $|G|=|\mathcal{U}(\C)|$. In this case, equality (\ref{order}) implies that $|\Gamma|=2$. Hence $\C$ is an extension of a rank $2$ pointed fusion category. The result then follows from \cite[Theorem 5.11]{Dong2018extensions}.
\end{proof}

In fact, Remark \ref{T_Y}(1) implies that Lemma \ref{slightly-deg1} classifies  slightly degenerate generalized Tambara-Yamagami fusion categories.

\begin{lemma}\label{slightly-deg2}
Let $\C$ be a braided generalized near-group fusion category of type $(G,\Gamma,k_1,\cdots,k_n)$. Assume that $(k_1,\cdots,k_n)\neq(0,\cdots,0)$ and $\C$ is slightly degenerate. Then $\C$ is exactly one of the following.

(1)\, $\C\cong \C_{ad}\boxtimes \C_{pt}$,  where $\C_{ad}$ is a Yang-Lee category.

(2)\, $\C\cong\C_{ad}\boxtimes \B$,  where $\C_{ad}$ is a slightly degenerate fusion category of the form $\C(\mathfrak{psl}_2,q^t,8)$ with $q=e^{\frac{\pi i}{8}}$ and $(t,2)=1$, $\B$ is a non-degenerate pointed fusion category.
\end{lemma}
\begin{proof}
By Proposition \ref{Cpt_of_sligdegen}, $\FPdim(\C_{pt})=|\mathcal{U}(\C)|$ or $\FPdim(\C_{pt})=2|\mathcal{U}(\C)|$.

Case $\FPdim(\C_{pt})=|\mathcal{U}(\C)|$. In this case, $\C_{ad}$ is a Yang-Lee category by Lemma \ref{dimuc}. Hence $\C\cong \C_{ad}\boxtimes \C_{ad}'$ by Theorem \ref{MugerThm}, where $\C_{ad}'=\C_{pt}$ by \cite[Corollary 3.29]{drinfeld2010braided}. Hence $\C\cong \C_{ad}\boxtimes \C_{pt}$. This proves Part (1).

\medbreak
Case $\FPdim(\C_{pt})=2|\mathcal{U}(\C)|$. By Theorem \ref{categorytype}, every component $\C_g$ of the universal grading of $\C$ at least has one invertible simple object. Moreover, every component $\C_g$ admits the same type. Hence every component $\C_g$ exactly contains two invertible simple objects.

By Proposition \ref{normalsubgroup}, the number of non-isomorphic non-invertible simple objects is not more than the order of $G$. Hence the number of non-isomorphic non-invertible simple objects in $\C_g$ is $1$ or $2$.

\medbreak
If the first case holds true then $\C_{ad}$ is a fusion category of rank $3$. By Lemma \ref{Cpt_of_sligdegen}, the M\"{u}ger center of $\C_{ad}$ contains the category $\svect$. This contradicts Lemma \ref{slight_degenerate} which says that the rank of $\C_{ad}$ should be even.

\medbreak
If the second case holds true then $\C_{ad}$ is a rank $4$ fusion category. Let $\delta$ be the non-trivial invertible simple object in $\C_{ad}$, and $Y_1,Y_2$ be the non-invertible simple objects in $\C_{ad}$. Then $\delta$ generates the category $\svect$ by Lemma \ref{Cpt_of_sligdegen}(2). By \cite[Lemma 5.4]{muger2000galois}, $\delta\otimes Y_i$ is not isomorphic to $Y_i$ for $i=1,2$. Hence $G[Y_i]$ is trivial and $\delta\otimes Y_i\cong Y_j$ for $i\neq j$.

The fact obtained above implies that if the M\"{u}ger center $\Z_2(\C_{ad})$ of $\C_{ad}$ contains $Y_1$ or $Y_2$ then $\Z_2(\C_{ad})=\C_{ad}$ and hence $\C_{ad}$ is symmetric. Since $\C_{ad}$ contains $\svect$, $\C_{ad}$ is not Tannakian. In addition, $\FPdim(\C_{ad})>2$. Hence $\C_{ad}$ should admit a $\mathbb{Z}_2$-extension of a Tannakian subcategory by \cite[Corollary 2.50]{drinfeld2010braided}. This contradicts Remark \ref{remark1} which says the universal grading group of $\C_{ad}$ is trivial. This proves that $\Z_2(\C_{ad})=\svect$ and hence $\C_{ad}$ is slightly degenerate. By \cite[Theorem 3.1]{bruillard2017classification}, $\C_{ad}$ is a fusion category of the form $\C(\mathfrak{psl}_2,q^t,8)$ with $q=e^{\frac{\pi i}{8}}$ and $(t,2)=1$.

By Lemma \ref{Cpt_of_sligdegen}(2) and the arguments above, $\mathcal{Z}_2(\C_{ad})=\mathcal{Z}_2(\C_{ad}^{'})=\svect$. On the other hand, \cite[Proposition 3.29]{drinfeld2010braided} shows that $\C_{ad}^{'}=\C_{pt}$. Hence $\C_{pt}$ is slightly degenerate and admits a decomposition $\C_{pt}\cong\svect\boxtimes\B$ by \cite[Proposition 2.6(ii)]{etingof2011weakly}, where $\B$ is a non-degenerate pointed fusion category. So $\C$ admits a decomposition $\C\cong\B\boxtimes\B'$ by Theorem \ref{MugerThm}. Counting rank and Frobenius-Perron dimensions of simple objects on both sides, we obtain that $\B'$ is a rank $4$ non-pointed fusion category. By Remark \ref{remark1}, $\C_{ad}$ is the smallest non-pointed fusion subcategory of $\C$. Hence $\C_{ad}=\B'$. This proves Part (2).
\end{proof}

Combing Lemma \ref{slightly-deg1} and \ref{slightly-deg2}, we obtain the classification of slightly degenerate generalized near-group fusion categories.

\begin{theorem}
Let $\C$ be a slightly degenerate generalized near-group fusion category. Then $\C$ is exactly one of the following::

(1)\,  $\C\cong \I\boxtimes \B$, where $\I$ is an Ising category, $\B$ is a slightly degenerate pointed fusion category.

(2)\, $\C\cong \C_{ad}\boxtimes \C_{pt}$,  where $\C_{ad}$ is a Yang-Lee category.

(3)\, $\C\cong\C_{ad}\boxtimes \B$,  where $\C_{ad}$ is a slightly degenerate fusion category of the form $\C(\mathfrak{psl}_2,q^t,8)$ with $q=e^{\frac{\pi i}{8}}$ and $(t,8)=1$, $\B$ is a non-degenerate pointed fusion category.

(4)\,  $\C$ is generated by a $\sqrt{2}$-dimensional simple object. In this case, $\C$ is prime.
\end{theorem}

\section*{Acknowledgements}
The research of the author is partially supported by the startup foundation for introducing talent of NUIST (Grant No. 2018R039) and the Natural Science Foundation of China (Grant No. 11201231).



\end{document}